\documentclass[a4paper,11pt]{amsart}
\usepackage{geometry}
\usepackage{amsmath,amssymb,amsfonts,latexsym}
\usepackage{graphicx}
\usepackage{algorithm}
\usepackage{algorithmicx}
\usepackage{algpseudocode}
\usepackage[table]{xcolor}
\usepackage{float}
\usepackage{graphics}
\usepackage{subcaption}
\usepackage{listings,xcolor}

\newtheorem{theorem}{Theorem}[section]

\newtheorem{corollary}[theorem]{Corollary}
\newtheorem{definition}[theorem]{Definition}

\newtheorem{remark}{Remark}

\begin{document}
\title[Construction of generalization of the Leibnitz numbers...]{%
{\Large Construction of a generalization of the Leibnitz numbers and their properties}}

\author{YILMAZ SIMSEK}
\address{Department of Mathematics, Faculty of Science University of Akdeniz
TR-07058 Antalya-TURKEY}
\email{ysimsek@akdeniz.edu.tr}

\begin{abstract}
The aim of this paper is to give a novel generalization of the Leibnitz numbers derived from application of the Beta function to the modification for the Bernstein basis functions. We also give some properties of the Leibnitz numbers with the aid of their generating functions derived from the Volkenborn integral on the set of $p$-adic integers. We also give some novel identities and relations involving the Leibnitz numbers, the Daehee numbers, the Changhee numbers, inverse binomial coefficients, and combinatorial sums. Finally, by coding computation formula for the generalization of the Leibnitz numbers in Mathematica 12.0 with their implementation, we compute few values of these numbers with their tables. Finally, by using the applications of Volkenborn integral to Mahler coefficients, we derive some novel formulas involving the Leibnitz numbers.

\vspace{2mm}


\vspace{2mm}

\end{abstract}

\thanks{\textbf{The present paper was supported by the \textit{Scientific Research Project Administration of Akdeniz University.}}}	
			
\maketitle

%\thanks{\textit{}}

%\newline

%%%%%%%%%%%%%%%%%%%%%%%%%%%%%%%%%%%%%%%%%%%%%%%%%%%%%%%%%%%%%%%%%%%

%%%%%%%%%%%%%%%%%%%%%%%%%%%%%%%%%%%%%%%%%%%%%%%%%%%%%%%%%%%%%%%%%%%

%%%%%%%%%%%%%%%%%%%%%%%%%%%%%%%%%%%%%%%%%%%%%%%%%%%%%%%%%%%%%%%%%%%

\section{Introduction}

%%%%%%%%%%%%%%%%%%%%%%%%%%%%%%%%%%%%%%%%%%%%%%%%%%%%%%%%%%%%%%%%%%%

\subsection{About Professor Lee Chae Jang:}

I met Professor Lee Chae Jang at the international conference organized by
the Jangjeon Mathematical Society in Mysore, India, in 2003. Professor
Taekyun Kim, Professor Seog-Hoon Rim, and Professor Lee Chae Jang are
founders of the Jangjeon Mathematics Society. This society has been not only
published two international mathematical journals including the Advanced
Studies in Contemporary Mathematics (ASCM) and the Proceedings of the
Jangjeon Mathematical Society (PJMS), but also organized International
Conference of The Jangjeon Mathematical Society. Until I met Professor Jang,
he has always contributed to mathematics in a disciplined manner and has
always been an example for young scientists, without losing his passion for
mathematics. I had a lot of interactions and views with Professor Jang, both
culturally and from a mathematical point of view.

\textit{For Professor Jang's great contribution to mathematics and his
	guidance for young scientists;}

\textit{I dedicate this article to his retirement with my best wishes.}

\textit{I wish the rest of his life to be happy, fruitful, success and
	healthy with all his family and friends.}

It is known by many researchers, who work on the subject of special numbers and their applications in recent years, that this subject has became among the leading topics of mathematics and especially analytic number theory. The \textit{Leibnitz numbers}, known as the famous German mathematician Gottfried Wilhelm Leibnitz (1646 - 1716), are considered in this paper. These numbers, which have rarely been addressed until now, are studied by using the techniques of generating functions and their Volkenborn integral representation in this paper. These numbers are also closely related to the \textit{Leibniz Harmoic Triangle} numbers. The denominators of some of these numbers are also directly related to the \textit{pronic numbers}. Within the scope of this study, it has been proved that these numbers are also related to Daehee numbers, Changhee numbers and the combinatorial numbers we found.
\subsection{Definitions and Notations}

Let $\mathbb{N}$ and $\mathbb{C}$ denote the set of natural numbers and the set of complex numbers, respectively, and also $\mathbb{N}_{0}=%
\mathbb{N}\cup \left\{ 0\right\} $.

 Let $x_{(n)}=x\left(x-1\right)\dots\left(x-n+1\right)$ with $x_{(0)}=1$ and $n\in \mathbb{N}$.

The Leibnitz numbers, $\boldsymbol{l}\left( n,k\right) $, are defined by%%
\begin{equation}
\boldsymbol{l}\left( n,k\right) =\frac{1}{\left( n+1\right) \binom{n}{k}}
\label{ExpLeib}
\end{equation}%
whose generating function is given as follows:%
\begin{equation}
\mathcal{G}_{l}\left( t,u\right) :=\sum\limits_{n=0}^{\infty
}\sum\limits_{k=0}^{n}\boldsymbol{l}\left( n,k\right) t^{k}u^{n}=\frac{\log
	\left( 1-u\right) +\log \left( 1-ut\right) }{\left( 1-u\right) \left(
	1-tu\right) -1},  \label{A1}
\end{equation}%
where $|u|<1$; $k=0,1,2,\ldots ,n$ and $n\in \mathbb{N}_{0}$ (\textit{cf}. \cite[Exercise 16, p. 127]{Charalambos}).

As seen from the equation (\ref{A1}), the function $\mathcal{G}_{l}\left( t,u\right)$ is the generating function for the polynomials: $$L_{n}(t):=\sum\limits_{k=0}^{n}\boldsymbol{l}\left( n,k\right) t^{k}$$ whose coefficients are the Leibnitz numbers and also whose degree is $n$.
That is, the ordinary generating function for the polynomials $L_{n}(t)$ is given as follows:
\begin{equation}
\mathcal{G}_{l}\left( t,u\right)=\sum_{n=0}^{\infty}L_{n}(t)u^n.
\label{GF-P}
\end{equation}

Observe that
\begin{eqnarray*}
L_{n}(1)=\sum\limits_{k=0}^{n}\boldsymbol{l}\left( n,k\right),
\end{eqnarray*}
and
\begin{eqnarray*}
	L_{n}(0)=\boldsymbol{l}\left( n,0\right)=\frac{1}{n+1}.
\end{eqnarray*}

The Leibnitz numbers, $\boldsymbol{l}\left( n,k\right)$, are also given by
following finite combinatorial sum:%
\begin{equation}
\boldsymbol{l}\left( n,k\right) =\sum\limits_{v=0}^{k}\left( -1\right) ^{k-v}%
\frac{1}{n-v+1}\binom{k}{v},  \label{SumLeib}
\end{equation}%
where $k=0,1,2,\ldots ,n$ and $n\in \mathbb{N}_{0}$ (\textit{cf}. \cite[Exercise
16, p. 127]{Charalambos}).

With the initial condition%
\begin{equation*}
\boldsymbol{l}\left( n,0\right) =\frac{1}{n+1},
\end{equation*}%
the Leibnitz numbers satisfy the following recurrence relation: 
\begin{equation*}
\boldsymbol{l}\left( n,k\right) =\frac{k}{n+1}\boldsymbol{l}\left(
n-1,k-1\right) ,
\end{equation*}%
where $k=1,2,\ldots ,n$ and $n\in \mathbb{N}$ (\textit{cf}. \cite[Exercise 16,
p. 127]{Charalambos}).

In \cite{LeibnitzZhao2010}, whose content could not be reached but whose existence is known, Zhao and Wuyungaowa claimed in its abstract that they gave a series of identities involving Leibniz numbers, Stirling numbers, harmonic numbers, and arctan numbers by making use of generating functions. They also claimed that give the asymptotic expansion of certain sums related to Leibniz numbers by the Laplace’s method. On the other hand, there is no data to comment on whether the Leibniz numbers mentioned there and the \textit{Leibnitz numbers} discussed in this study point to the same concept or the relationship of the results. Due to the expression \textit{Leibniz numbers} in the abstract of the relevant study, we cite it here. 

The Daehee numbers, $D_{n}$, are defined by%
\begin{equation}
\mathcal{G}_{D}\left( u\right) :=\frac{\log \left( 1+u\right) }{u}%
=\sum\limits_{n=0}^{\infty }D_{n}\frac{u^{n}}{n!}  \label{Dahee}
\end{equation}%
where 
\begin{equation}
D_{n}=(-1)^{n}\frac{n!}{n+1}  \label{D}
\end{equation}%
(\textit{cf}. \cite{KimDahee}).

The Changhee numbers, $Ch_{n}$, are defined by%
\begin{equation}
\mathcal{G}_{Ch}\left( u\right) :=\frac{2}{2+u}=\sum\limits_{n=0}^{\infty
}Ch_{n}\frac{u^{n}}{n!}  \label{Changhee}
\end{equation}%
where 
\begin{equation}
Ch_{n}=(-1)^{n}\frac{n!}{2^{n}}  \label{C}
\end{equation}%
(\textit{cf}. \cite{KimChange}).

The numbers, $Y_{n}\left( \lambda \right) $, are defined by 
\begin{equation}
\mathcal{G}_{Y}\left( u,\lambda \right) :=\frac{2}{\lambda \left( 1+\lambda
	u\right) -1}=\sum\limits_{n=0}^{\infty }Y_{n}\left( \lambda \right) \frac{%
	u^{n}}{n!}  \label{YPoly}
\end{equation}%
where 
\begin{equation}
Y_{n}\left( \lambda \right) =\left( -1\right) ^{n}\frac{2n!}{\lambda -1}%
\left( \frac{\lambda ^{2}}{\lambda -1}\right) ^{n}  \label{YExp}
\end{equation}%
(\textit{cf}. \cite{simsekTurkish}, \cite{KucukogluJNT2017}).

The relation between Changhee numbers and the numbers $Y_{n}\left(
\lambda\right)$ is given as follows: 
\begin{equation}
Ch_{n}=\left(-1\right)^{n+1} Y_n \left(-1\right)  \label{R-CYn}
\end{equation}
(\textit{cf}. \cite{simsekTurkish}, \cite{KucukogluJNT2017}).

Next, we summarize the content of this paper as follows:

In Section 2, we give combinatorial identities and relations related to the Leibnitz numbers, Daehee numbers and Changhee numbers. We also give some computational formulas for these numbers.

In Section 3, we give further remarks and observations on the Leibnitz numbers. Moreover, by using finite sums derived from application of the integral to
the modification for the Bernstein basis functions, we introduce a generalization of the Leibnitz numbers.

In Section 4, we give Mathematica implementation of the generalized Leibnitz numbers and by this implementation, we compute few values of these numbers with their tables.

In Section 5, by applying Volkenborn integral on the set of $p$-adic integers, we derive some novel formulas involving the Leibnitz numbers.

\section{Combinatorial identities and relations involving the Leibnitz numbers, Daehee numbers and Changhee numbers}
By using functional equations of the generating functions for the the Leibnitz numbers, the Daehee numbers and special series, we find many formulas, identities and relations involvilg the Changhee numbers and combinatorial numbers and sums.

By using (\ref{A1}), we get 
\begin{equation}
\sum\limits_{n=0}^{\infty }\sum\limits_{k=0}^{n}\boldsymbol{l}\left(
n,k\right) t^{k}u^{n} =\frac{1}{t+1}\sum_{n=0}^{\infty} \sum_{k=0}^{n} \frac{%
	1}{n-k+1} \left(\frac{t}{t+1}\right)^k \left(1+t^{n-k}\right)u^n.
\end{equation}
Comparing the coefficients $u^n$ on both-sides of the above equation, we get
the following theorem:

\begin{theorem}
	\begin{equation}
	\sum\limits_{k=0}^{n}\boldsymbol{l}\left( n,k\right) t^{k}=\frac{1}{t+1}%
	\sum_{k=0}^{n} \frac{1}{n-k+1} \left(\frac{t}{t+1}\right)^k
	\left(1+t^{n-k}\right).  \label{Th-ik1}
	\end{equation}
\end{theorem}

Substituting $t=1$ into (\ref{Th-ik1}), we get the finite summation of the
Leibnitz numbers as in the following corollary:

\begin{corollary}
	\begin{equation}
	\sum\limits_{k=0}^{n}\boldsymbol{l}\left( n,k\right)=\sum_{k=0}^{n} \frac{1}{%
		\left(n-k+1\right)2^k}.  \label{Th-ik2}
	\end{equation}
\end{corollary}

\begin{theorem}
	\begin{equation}
	\sum\limits_{k=0}^{n}\boldsymbol{l}\left( n,k\right) t^{k}=\frac{1}{1+t}%
	\sum\limits_{j=0}^{n}\left( -1\right) ^{j}\frac{D_{j}}{j!}\left( \frac{t}{1+t%
	}\right) ^{n-j}\left( 1+t^{j+1}\right) .  \label{1a}
	\end{equation}
\end{theorem}

\begin{proof}
	By using (\ref{A1}) and (\ref{Dahee}), we get the following functional
	equation of generating functions: 
	\begin{equation}
	\mathcal{G}_{l}\left( t,u\right) =\frac{\mathcal{G}_{D}\left( -u\right) +t%
		\mathcal{G}_{D}\left( -ut\right) }{1+t-ut}
	\end{equation}%
	which yields 
	\begin{eqnarray*}
		\sum\limits_{n=0}^{\infty }\sum\limits_{k=0}^{n}\boldsymbol{l}\left(
		n,k\right) t^{k}u^{n} &=&\frac{1}{1+t}\sum\limits_{n=0}^{\infty }u^{n}\frac{%
			t^{n}}{\left( 1+t\right) ^{n}}\left( \sum\limits_{n=0}^{\infty }\frac{\left(
			-1\right) ^{n}D_{n}}{n!}u^{n}+t\sum\limits_{n=0}^{\infty }\frac{\left(
			-1\right) ^{n}D_{n}u^{n}}{n!}t^{n}\right).
	\end{eqnarray*}
Hence, we get
\begin{eqnarray*}
	\sum\limits_{n=0}^{\infty }\sum\limits_{k=0}^{n}\boldsymbol{l}\left(
	n,k\right) t^{k}u^{n} &=&\frac{1}{1+t}\sum\limits_{n=0}^{\infty }\sum\limits_{j=0}^{n}\left(
	-1\right) ^{j}\frac{D_{j}}{j!}\left( \frac{t}{1+t}\right) ^{n-j}u^{n} \\
	&&+\frac{t}{1+t}\sum\limits_{n=0}^{\infty }\sum\limits_{j=0}^{n}\left(
	-1\right) ^{j}\frac{D_{j}}{j!}t^{j}\left( \frac{t}{1+t}\right) ^{n-j}u^{n}.
\end{eqnarray*}
	Comparing the coefficients of $u^{n}$ on both sides of the above equation,
	we arrive at the desired result.
\end{proof}

Combining (\ref{1a}) with (\ref{D}), we have the following corollary:

\begin{corollary}
	\begin{equation}
	\sum\limits_{k=0}^{n}\boldsymbol{l}\left( n,k\right) t^{k}=\frac{1}{1+t}%
	\sum\limits_{j=0}^{n}\frac{1}{j+1}\left( \frac{t}{1+t}\right) ^{n-j}\left(
	1+t^{j+1}\right) .  \label{2}
	\end{equation}
\end{corollary}

Observe that the equation (\ref{2}) is equivalent to the equation (\ref%
{Th-ik1}).

In the next, by combining (\ref{2}) with (\ref{ExpLeib}), we also have the
following corollary:

\begin{corollary}
	\begin{equation}
	\sum\limits_{k=0}^{n}\frac{t^{k}}{\left( n+1\right) \binom{n}{k}}=\frac{1}{%
		1+t}\sum\limits_{j=0}^{n}\frac{1+t^{j+1}}{j+1}\left( \frac{t}{1+t}\right)
	^{n-j}.  \label{Cor-1}
	\end{equation}
\end{corollary}

Substituting $t=1$ into (\ref{Cor-1}) yields the following result:

\begin{corollary}
	\begin{equation}
	\sum\limits_{k=0}^{n}\frac{1}{\left( n+1\right) \binom{n}{k}}%
	=\sum\limits_{j=0}^{n}\frac{2^{j-n}}{j+1}.  \label{3}
	\end{equation}
\end{corollary}

Observe that the combination of (\ref{ExpLeib}) with equation (\ref{3}) is
equivalent to equation (\ref{Th-ik2}).

Combining (\ref{3}) with (\ref{C}), we have the following result:

\begin{theorem}
	\begin{equation}
	\sum\limits_{k=0}^{n}\frac{1}{\left( n+1\right) \binom{n}{k}}%
	=\sum\limits_{j=0}^{n}\frac{\left( -1\right) ^{n-j}\left( n-j\right) !}{%
		\left( j+1\right) Ch_{n-j}}.  \label{Th-ik3}
	\end{equation}
\end{theorem}

Substituting (\ref{R-CYn}) into (\ref{Th-ik3}), we arrive at the following
corollary:

\begin{corollary}
	\begin{equation*}
	\sum\limits_{k=0}^{n}\frac{1}{\left( n+1\right) \binom{n}{k}}%
	=-\sum\limits_{j=0}^{n}\frac{\left( n-j\right) !}{\left( j+1\right)
		Y_{n-j}\left( -1\right) }.
	\end{equation*}
\end{corollary}

Substituting $t=1$ into (\ref{A1}) and combining the final eqaution with (%
\ref{Dahee}), we get the following functional equation:

\begin{theorem}
	Let $n\in \mathbb{N}$. Then we have%
	\begin{equation}
	\sum\limits_{k=0}^{n}\boldsymbol{l}\left( n,k\right) -\frac{1}{2}%
	\sum\limits_{k=0}^{n-1}\boldsymbol{l}\left( n-1,k\right) =\frac{(-1)^{n}}{n!}%
	D_{n}.  \label{k1}
	\end{equation}
\end{theorem}

\begin{proof}
	We set the following functional equation involving generating functions for
	the Leibnitz numbers and the Daehee numbers:%
	\begin{equation}
	\mathcal{G}_{D}\left( -u\right) =\left( 1-\frac{u}{2}\right) \mathcal{G}%
	_{l}\left( 1,u\right) .  \label{fe}
	\end{equation}%
	By using the above equation, we obtain%
	\begin{equation*}
	\sum\limits_{n=0}^{\infty }\sum\limits_{k=0}^{n}\boldsymbol{l}\left(
	n,k\right) u^{n}-\frac{1}{2}\sum\limits_{n=0}^{\infty }\sum\limits_{k=0}^{n}%
	\boldsymbol{l}\left( n,k\right) u^{n+1}=\sum\limits_{n=0}^{\infty
	}(-1)^{n}D_{n}\frac{u^{n}}{n!}.
	\end{equation*}%
	Therefore%
	\begin{equation*}
	1+\sum\limits_{n=1}^{\infty }\sum\limits_{k=0}^{n}\boldsymbol{l}\left(
	n,k\right) u^{n}-\frac{1}{2}\sum\limits_{n=1}^{\infty
	}\sum\limits_{k=0}^{n-1}\boldsymbol{l}\left( n-1,k\right)
	u^{n}=D_{0}+\sum\limits_{n=0}^{\infty }(-1)^{n}D_{n}\frac{u^{n}}{n!}.
	\end{equation*}%
	Comparing the coefficients of $u^{n}$ on both sides of the above equation
	yields the desired result.
\end{proof}

Combining (\ref{k1}) with (\ref{D}), we obtain the following result:

\begin{theorem}
	Let $n\in \mathbb{N}$. Then we have%
	\begin{equation*}
	\sum\limits_{k=0}^{n}\boldsymbol{l}\left( n,k\right) -\frac{1}{2}%
	\sum\limits_{k=0}^{n-1}\boldsymbol{l}\left( n-1,k\right) =\frac{1}{n+1}.
	\end{equation*}
\end{theorem}

\begin{remark}
	By using (\ref{fe}), assuming that $\left\vert u\right\vert <1$, we obtain%
	\begin{equation*}
	\sum\limits_{n=0}^{\infty }\sum\limits_{k=0}^{n}\boldsymbol{l}\left(
	n,k\right) u^{n}=\sum\limits_{n=0}^{\infty }\frac{u^{n}}{n+1}%
	\sum\limits_{n=0}^{\infty }\frac{u^{n}}{2^{n}}.
	\end{equation*}%
	Therefore%
	\begin{equation*}
	\sum\limits_{n=0}^{\infty }\sum\limits_{k=0}^{n}\boldsymbol{l}\left(
	n,k\right) u^{n}=\sum\limits_{n=0}^{\infty }\sum\limits_{k=0}^{n}\frac{%
		2^{k-n}}{k+1}u^{n}.
	\end{equation*}%
	Comparing the coefficients of $u^{n}$ on both sides of the above equation
	yields the Eq.-\textup{(\ref{3})}.
\end{remark}

\section{Generalized Leibnitz numbers}

Here, we give further remarks and observations on the Leibnitz numbers and
their relations with finite sums derived from application of the integral to
the modification for the Bernstein basis functions. By using the following
well-known beta functions:%
\begin{equation*}
B\left( \alpha ,\beta \right) =\int\limits_{0}^{1}t^{\alpha -1}\left(
1-t\right) ^{\beta -1}dt=B\left( \beta ,\alpha \right)
\end{equation*}%
where $\operatorname{Re}\left( \alpha \right) >0$ and $\operatorname{Re}\left( \beta \right)>0$ (\textit{cf}. \cite{raiville}, \cite{SrivatavaChoi}), one has the
following novel integral formula:%
\begin{equation}
\int\limits_{a}^{b}\left( x-a\right) ^{\alpha -1}\left( b-x\right) ^{\beta
	-1}dx=\left( b-a\right) ^{\alpha +\beta -1}B\left( \alpha ,\beta \right)
\label{IK}
\end{equation}%
(\textit{cf}. \cite[p.10, Eq. (69)]{SrivatavaChoi}).

Let $x\in \left[ a,b\right] $ and $k=0,1,2,\ldots ,n$. By combining (\ref{IK}%
) with the modification for the Bernstein basis functions is given as
follows:%
\begin{equation*}
B_{k}^{n}\left( x;a,b\right) =\binom{n}{k}\left( \frac{x-a}{b-a}\right)
^{k}\left( \frac{b-x}{b-a}\right) ^{n-k}\text{;\ \ \ }a<b
\end{equation*}%
where (\textit{cf}. \cite{Goldman}, \cite{SimsekCatalan}; see references
therein), we have%
\begin{equation*}
\int\limits_{a}^{b}B_{k}^{n}\left( x;a,b\right) dx=\binom{n}{k}%
\sum\limits_{j=0}^{k}\sum\limits_{l=0}^{n-k}\left( -1\right) ^{n-j-l}\binom{k%
}{j}\binom{n-k}{l}\frac{a^{k-j}b^{n+j-k+1}-a^{n-l+1}b^{l}}{n+j-k-l+1}
\end{equation*}%
Thus, by the above integral representation, we also get%
\begin{eqnarray}
&&\sum\limits_{j=0}^{k}\sum\limits_{v=0}^{n-k}\left( -1\right) ^{n-j-v}%
\binom{k}{j}\binom{n-k}{v}\frac{a^{k-j}b^{n+j-k+1}-a^{n-v+1}b^{v}}{n+j-k-v+1}
\label{GLNum} \\
&=&\left( b-a\right) B\left( k+1,n-k+1\right),  \notag
\end{eqnarray}
(\textit{cf}. \cite[Eq. (28)]{SimsekCatalan}).

By using (\ref{GLNum}), it is time to introduce generalized Leibnitz numbers
by the following definition:

\begin{definition}
	Let $a$ and $b$ be nonnegative real parameters with $a\neq b$. Let $%
	k=0,1,2,\ldots ,n$ and $n\in\mathbb{N}_{0}$. Then, the generalized Leibnitz
	numbers $\mathcal{L}\left( n,k;a,b\right)$ are defined by 
	\begin{equation}
	\mathcal{L}\left( n,k;a,b\right)
	=\sum\limits_{j=0}^{k}\sum\limits_{v=0}^{n-k}\left( -1\right) ^{n-j-v}\binom{%
		k}{j}\binom{n-k}{v}\frac{a^{k-j}b^{n+j-k+1}-a^{n-v+1}b^{v}}{n+j-k-v+1}.
	\label{GLNumDef}
	\end{equation}
\end{definition}

Note that in (\ref{GLNumDef}) we assume that $0^0=1$.

\begin{remark}
	Setting $a=0$ and $b=1$ in \textup{(\ref{GLNumDef})} and using \textup{(\ref%
		{ExpLeib})} yields 
	\begin{equation*}
	\frac{1}{\left(n+1\right)\binom{n}{k}} =\sum\limits_{v=0}^{n-k}\left(
	-1\right) ^{n-k-v}\binom{n-k}{v}\frac{1}{n-v+1},
	\end{equation*}
	(\textit{cf}. \cite{SimsekBoundary, SimsekHacettepe, SimsekAMC2015}, \cite[%
	Eq. (29)]{SimsekCatalan}), which is equivalent to \textup{(\ref{SumLeib})}.
	
	Observe that further identities and new number families may be discovered by
	using the methods in paper \cite{SimsekHacettepe, SimsekBoundary,
		SimsekAMC2015,SimsekCatalan}, of the generalized Leibnitz numbers, which
	seem to be closely related to the modification of the Bernstein basis
	functions.
\end{remark}

\textbf{Open Problems:}

\begin{itemize}
	\item How can we constuct generating functions for the (generalized)
	Leibnitz numbers?
	
	\item Is there any zeta-type functions, on the set of complex numbers, which
	interpolates the (generalized) Leibnitz numbers?
\end{itemize}

\section{Mathematica implementation of the numbers $\mathcal{L}\left(
	n,k;a,b\right) $}

In this section, in order to give some applications of the computation formulas given in the previous sections, we present Mathematica implementation (see: Implementation \ref{CodeGLeibnitzNum}) for the numbers $\mathcal{L}\left(
n,k;a,b\right) $ by coding (\ref{GLNumDef}) in Mathematica 12.0.

\begin{lstlisting}[language=Mathematica, label=CodeGLeibnitzNum, caption={The following Mathematica code returns the values of the generalized Leibnitz numbers $\mathcal{L}\left( n,k;a,b\right)$}.]
Unprotect[Power];
Power[0,0]=1;
Protect[Power];
GLeibnitzNum[n_,k_,a_,b_]:=Sum[Sum[((-1)^(n-j-v))*Binomial[k,j]*Binomial[n-k,v]*((a^(k-j))*(b^(n+j-k+1))-(a^(n-v+1))(b^v))/(n+j-k-v+1), {v,0,n-k}], {j,0,k}]
\end{lstlisting}

Then, by (\ref{CodeGLeibnitzNum}), we compute few values of the numbers 
$\mathcal{L}\left( n,k;a,b\right) $, and give their tables as follows:

\begin{table}[H]
	\caption{\textmd{For $k=1$ and $n \in \{0, 1,2,3,4\}$, few values of the
			generalized Leibnitz numbers $\mathcal{L}\left( n,k;a,b\right)$ .}}
	\label{TableGLeibnitzNum1}\includegraphics[width=%
	\textwidth]{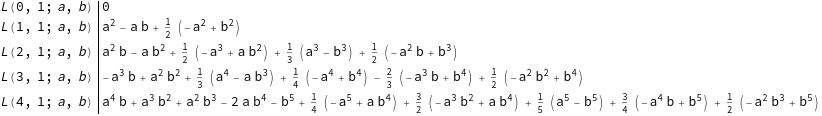}
\end{table}

\begin{table}[H]
	\caption{\textmd{For $k=2$ and $n \in \{0, 1,2,3,4\}$, few values of the
			generalized Leibnitz numbers $\mathcal{L}\left( n,k;a,b\right)$ .}}
	\label{TableGLeibnitzNum2}\includegraphics[width=%
	\textwidth]{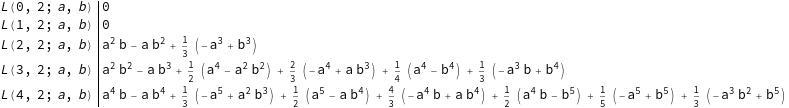}
\end{table}

\begin{table}[H]
	\caption{\textmd{For $n \in \{0,1,2,3\}$ and $k \in \{0,1,2\}$, few values
			of the generalized Leibnitz numbers $\mathcal{L}\left( n,k;a,b\right)$.}}
	\label{MatrixTableGLeibnitzNum1}%
	\includegraphics[scale=0.5]{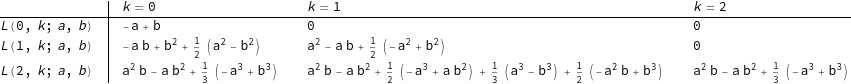}
\end{table}

Substituting $a=0$ and $b=1$, the numbers 
$\mathcal{L}\left( n,k;a,b\right) $ are reduced to the classical Leibnitz numbers $\boldsymbol{l}\left( n,k\right)$.

\begin{table}[H]
	\caption{\textmd{For $n \in \{0,1,2,\dots,8\}$ and $k \in \{0,1,2,\dots,8\}$, few values
			of the generalized Leibnitz numbers $\mathcal{L}\left( n,k;0,1\right)$, namely $\boldsymbol{l}\left( n,k\right)$.}}
	\label{MatrixTableLeibnitzNum}%
	\includegraphics[scale=0.5]{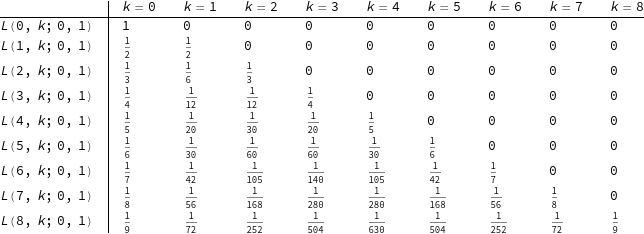}
\end{table}

\section{The Volkenborn integral representation of the Leibnitz numbers on the set of $p$-adic integers}

In this section, by applying Volkenborn integral on the set of $p$-adic integers not only to the Mahler coefficients, but also to uniformly differential function on the set of $p$-adic integers, we obtain some novel formulas involving the Leibnitz numbers.

Here, we follow notations of the following references: \cite{KimV,
	Kim2006TMIC, Schikof, SimsekMTJPM}; and the references cited therein.

Some notations and definitions for $p$-adic integrals are given as follows:

Let $m\in \mathbb{N}$. Let $ord_{p}(m)$ denote the greatest integer $k$ ($%
k\in \mathbb{N}_{0})$ such that $p^{k}$ divides $m$ in $\mathbb{Z}$. If $m=0$%
, then $ord_{p}(m)=\infty $.

$\left\vert .\right\vert _{p}$ is a norm on $\mathbb{Q}$. This norm is given
by%
\begin{equation*}
\left\vert x\right\vert _{p}=\left\{ 
\begin{array}{cc}
p^{-ord_{p}(x)} & \text{if }x\neq 0, \\ 
0 & \text{if }x=0.%
\end{array}%
\right.
\end{equation*}%
Let $\mathbb{Z}_{p}$ be a set of $p$-adic integers which is given by%
\begin{equation*}
\mathbb{Z}_{p}=\left\{ x\in \mathbb{Q}_{p}:\left\vert x\right\vert _{p}\leq
1\right\} .
\end{equation*}

Let $f$ define on $\mathbb{Z}_{p}$. The function $f$ is called a uniformly
differential function at a point$\ a\in \mathbb{Z}_{p}$ if $f$ satisfies the
following conditions:

If the difference quotients $\Phi _{f}:\mathbb{Z}_{p}\times \mathbb{Z}%
_{p}\rightarrow \mathbb{C}_{p}$ such that%
\begin{equation*}
\Phi _{f}(x,y)=\frac{f(x)-f(y)}{x-y}
\end{equation*}%
have a limit $f^{\prime }(z)$ as $(x,y)\rightarrow (0,0)$ ($x$ and $y$
remaining distinct). A set of uniformly differential functions is briefly
indicated by $f\in UD(\mathbb{Z}_{p})$ or $f\in C^{1}(\mathbb{Z}%
_{p}\rightarrow \mathbb{C}_{p}\mathbb{)}$.

The well-known Volkenborn integral (bosonic $p$-integral) is given by%
\begin{equation*}
\int\limits_{\mathbb{Z}_{p}}f\left( x\right) d\mu _{1}\left( x\right) =%
\underset{N\rightarrow \infty }{\lim }\frac{1}{p^{N}}\sum_{x=0}^{p^{N}-1}f%
\left( x\right) ,
\end{equation*}%
where denotes the Haar distribution, which is defined by%
\begin{equation*}
\mu _{1}\left( x\right) =\mu _{1}\left( x+p^{N}\mathbb{Z}_{p}\right) =\frac{1%
}{p^{N}}
\end{equation*}%
(\textit{cf}. \cite{Jangsimsek}-\cite{KimChange}, \cite{Schikof}; see also
the references cited in each of these earlier works).

In order to achieve the results of this section, we let
\begin{equation}
f\left(x\right) =\sum\limits_{n=0}^{\infty }a_{n}\binom{x}{n}\in C^{1}(%
\mathbb{Z}_{p}\rightarrow \mathbb{K)},
\label{Mahler1}
\end{equation}%
where $\binom{x}{n}=\frac{x_{(n)}}{n!}$ denotes the Mahler coefficients. Applying
the Volkenborn integral to the function $f\left(x\right)$ in terms of the
Mahler coefficients $\binom{x}{n}$, we have the following well-known formula:%
\begin{equation}
\int\limits_{\mathbb{Z}_{p}}f\left( x\right) d\mu _{1}\left( x\right)
=\sum\limits_{n=0}^{\infty }\frac{(-1)^{n}}{n+1}a_{n},
\label{Mahler2}
\end{equation}%
(\textit{cf}. \cite[p. 168-Proposition 55.3]{Schikof}).

In order to give generating function (\ref{A1}), we apply the Volkenborn
integral to the following uniformly differential function on $%
\mathbb{Z}_{p}$:

\begin{equation*}
f\left( x,u;t\right) =\left( 1-u\right) ^{x}\left( 1-tu\right) ^{x},
\end{equation*}%
where $u$, $x\in \mathbb{Z}_{p}$.

Substituting the above function into the following well-known integral
equation:%
\begin{equation*}
\int\limits_{\mathbb{Z}_{p}}f(x+1)d\mu _{1}\left( x\right) =\int\limits_{%
	\mathbb{Z}_{p}}f(x)d\mu _{1}\left( x\right) +f^{\prime }(k)
\end{equation*}%
(\textit{cf}. \cite{Schikof}), we get%
\begin{equation}
\int\limits_{%
	%TCIMACRO{\U{2124} }%
	%BeginExpansion
	\mathbb{Z}
	%EndExpansion
	_{p}}\left( 1-u\right) ^{x}\left( 1-tu\right) ^{x}d\mu _{1}\left( x\right) =%
\frac{\log \left[ \left( 1-u\right) \left( 1-tu\right) \right] }{\left(
	1-u\right) \left( 1-tu\right) -1}.  \label{A0}
\end{equation}

Consequently, the function on the right of equation (\ref{A0}) gives the
generating function given in equation (\ref{A1}) for the Leibnitz numbers.

Combining (\ref{A0}) with the binomial series
\begin{equation*}
\sum\limits_{n=0}^{\infty }x_{(n)}\frac{u^{n}}{n!}=\left( 1+u\right) ^{x},
\end{equation*}
and using (\ref{Mahler2}), we obtain%
\begin{equation*}
\sum\limits_{n=0}^{\infty }(-1)^{n}\frac{u^{n}}{n!}\sum\limits_{k=0}^{n}%
\binom{n}{k}t^{n-k}\int\limits_{\mathbb{Z}_{p}}x_{(k)}x_{(n-k)}d\mu
_{1}\left( x\right) =\sum\limits_{n=0}^{\infty }\sum\limits_{k=0}^{n}%
\boldsymbol{l}\left( n,k\right) t^{k}u^{n}.
\end{equation*}%
Comparing the coefficients of $u^{n}$ on both sides of the above equation
yields the following theorem:

\begin{theorem}
	Let $n\in \mathbb{N}_{0}$. Then we have%
	\begin{equation}
	\frac{(-1)^{n}}{n!}\sum\limits_{k=0}^{n}\binom{n}{k}t^{n-k}\int\limits_{%
		\mathbb{Z}_{p}}x_{(k)}x_{(n-k)}d\mu _{1}\left( x\right)
	=\sum\limits_{k=0}^{n}\boldsymbol{l}\left( n,k\right) t^{k}.  \label{a11}
	\end{equation}
\end{theorem}

Combining (\ref{a11}) with the following formula:%
\begin{equation*}
\int\limits_{\mathbb{Z}_{p}}x_{(n)}x_{(m)}d\mu _{1}\left( x\right)
=\sum_{k=0}^{m}(-1)^{m+n-k}\binom{n}{k}\binom{m}{k}\frac{k!(n+m-k)!}{n+m-k+1}
\end{equation*}%
where $m,n\in \mathbb{N}_{0}$ (\textit{cf}. \cite{SimsekArxive}), we arrive at the
following theorem:

\begin{theorem}
	Let $n\in \mathbb{N}_{0}$. Then we have%
	\begin{equation}
	\sum\limits_{k=0}^{n}\boldsymbol{l}\left( n,k\right) t^{k}=\frac{1}{n!}
	\sum_{k=0}^{n}\sum_{j=0}^{n-k}\left(-1\right)^j \binom{n}{k} \binom{n-k}{j}\binom{k}{j} \frac{j!\left(n-j\right)!}{n-j+1}t^{n-k}. \label{ki1}
	\end{equation}
\end{theorem}
Substituting $t=1$ into (\ref{ki1}), we get the following corollary:
\begin{corollary}
	\begin{equation}
	\sum\limits_{k=0}^{n}\boldsymbol{l}\left( n,k\right) =\frac{1}{n!}
	\sum_{k=0}^{n}\sum_{j=0}^{n-k}\left(-1\right)^j \binom{n}{k} \binom{n-k}{j}\binom{k}{j} \frac{j!\left(n-j\right)!}{n-j+1}. \label{ki2}
	\end{equation}
\end{corollary}
%%%%%%%%%%%%%%%%%%%%%%%%%%%%%%%%%%%%%%%%%%%%%%%%%%%%%%%%%%%%%%%%

%%%%%%%%%%%%%%%%%%%%%%%%%%%%%%%%%%%%%%%%%%%%%%%%%%%%%%%%%%%%%%%%

%%%%%%%%%%%%%%%%%%%%%%%%%%%%%%%%%%%%%%%%%%%%%%%%%%%%%%%%%%%%%%%%%%%

\end{document}